\documentclass[12pt]{amsart}
\usepackage[mathscr]{eucal}
\usepackage{amsmath}
\usepackage[dvipdfmx]{graphicx}

\theoremstyle{plain}
	\newtheorem{theorem}{Theorem}[section]
	\newtheorem{lemma}[theorem]{Lemma}
	\newtheorem{corollary}[theorem]{Corollary}
	
	\newtheorem{proposition}[theorem]{Proposition}
	\newtheorem{remark}[theorem]{Remark}

\theoremstyle{plain}
	\newtheorem{maintheorem}{Theorem}

\def\R{\mathbb{R}}

\def\calF{\mathcal{F}}

\def\calZ{\mathcal{Z}}

\def\e{\varepsilon}
\def\6{\partial}
\def\8{\infty}
\def\tu{\tilde{u}}

\def\bw{\bar{w}}
\def\bu{\bar{u}}
\def\bx{\bar{x}}

\def\rddots#1{\cdot^{\cdot^{\cdot^{#1}}}}

\makeatletter
\newcommand{\xRightarrow}[2][]{%
\ext@arrow 0055{\Rightarrowfill@}{#1}{#2}%
}
\def\Rightarrowfill@{\arrowfill@\Relbar\Relbar\Rightarrow}
\newcommand{\xLeftarrow}[2][]{%
\ext@arrow 0055{\Leftarrowfill@}{#1}{#2}%
}
\def\Leftarrowfill@{\arrowfill@\Leftarrow\Relbar\Relbar}
\newcommand{\xLongleftrightarrow}[2][]{%
\ext@arrow 0055{\llrafill@}{#1}{#2}%
}
\def\llrafill@{\arrowfill@\Leftarrow\Relbar\Rightarrow}
\makeatother

\makeatletter
 \@addtoreset{equation}{section}
\makeatother

\begin{document}


\title[A limit equation and bifurcation diagrams]{A limit equation and bifurcation diagrams of\\
semilinear elliptic equations with\\
general supercritical growth}
\thanks{This work was supported by JSPS KAKENHI Grant Number 16K05225.}

\author{Yasuhito Miyamoto}
\address{Graduate School of Mathematical Sciences, The University of Tokyo,
3-8-1 Komaba, Meguro-ku, Tokyo 153-8914, Japan}
\email{miyamoto@ms.u-tokyo.ac.jp}

\begin{abstract}
We study radial solutions of the semilinear elliptic equation
\[
\Delta u+f(u)=0
\]
under rather general growth conditions on $f$.
We construct a radial singular solution and study the intersection number between the singular solution and a regular solution.
An application to bifurcation problems of elliptic Dirichlet problems is given.
To this end, we derive a certain limit equation from the original equation at infinity, using a generalized similarity transformation.
Through a generalized Cole-Hopf transformation, all the limit equations can be reduced into two typical cases, i.e., $\Delta u+u^p=0$ and $\Delta u+e^u=0$.
\end{abstract}
\date{\today}
\maketitle


\section{Introduction and Main results}
Let $N\ge 3$ and $r:=|x|$.
In this paper we construct a radial singular solution $u^*(r)$ of the elliptic equation
\begin{equation}\label{S1E1}
\Delta u+f(u)=0
\end{equation}
under rather general growth conditions, and study the intersection number of two radial solutions $\calZ_{(0,\8)}[u(\,\cdot\,,\rho)-u^*(\,\cdot\,)]$.
Here, $u(r,\rho)$ is the classical radial solution of (\ref{S1E1}), which satisfies
\begin{equation}\label{S1E2}
\begin{cases}
u''+\frac{N-1}{r}u'+f(u)=0, & r>0,\\
u(0)=\rho,\\
u'(0)=0,
\end{cases}
\end{equation}
and $\calZ_I[u_0(\,\cdot\,)-u_1(\,\cdot\,)]$ denotes the intersection number of the two functions $u_0(r)$ and $u_1(r)$ defined in an interval $I\subset\R$, i.e., $\calZ_I[u_0(\,\cdot\,)-u_1(\,\cdot\,)]=\sharp\{r\in I;\ u_0(r)=u_1(r)\}$.
By a radial singular solution $u^*(r)$ of (\ref{S1E1}) we mean that $u^*(r)$ is a classical solution of the equation
\begin{equation}\label{S1EQ}
u''+\frac{N-1}{r}u'+f(u)=0
\end{equation}
on $(0,r_0)$ for some $r_0>0$ and $\lim_{r\downarrow 0}u^*(r)=\8$.
We give two applications of the intersection number.

By $F(u)$ we define
\[
F(u):=\int_u^{\8}\frac{dt}{f(t)}.
\]
We assume the following:
\begin{multline}\label{f1}
\textrm{One of the following (\ref{f11}) or (\ref{f12}) holds:}\tag{f1}
\end{multline}
\begin{multline}\label{f11}\textrm{(a generalization of $u^p$)}\quad
f(u)\in C^1[0,\infty),\ f(u)>0\ \textrm{for}\ u>0,\ f(0)=0,\\
f(u)\in C^2(u_0,\8)\ \textrm{for some}\ u_0>0,\ 
\lim_{u\downarrow 0}F(u)=\8,\ \textrm{and}\ \lim_{u\to\8}F(u)=0,\tag{f1-1}
\end{multline}
\begin{multline}\label{f12}\textrm{(a generalization of $e^u$)}\quad
f(u)\in C^1(\R),\ f(u)>0\ \textrm{for}\ u\in\R,\\
f(u)\in C^2(u_0,\8)\ \textrm{for some}\ u_0>0,\ \lim_{u\to -\8}F(u)=\8,\ \textrm{and}\ \lim_{u\to\8}F(u)=0.\tag{f1-2}
\end{multline}
\begin{multline}\label{f2}
\textrm{There exists the limit $q:=\lim_{u\to\8}\frac{f'(u)^2}{f(u)f''(u)}$,}\\
\textrm{which is denoted by $q$ throughout the present paper, and this limit is in $(0,\8)$}.\tag{f2}
\end{multline}
Note that the inverse function of $F$, which is denoted by $F^{-1}(u)$, can be defined for $u>0$, because of (\ref{f1}).
We define the growth rate of $f$ by $p:=\lim_{u\to\8}uf'(u)/f(u)$.
By L'Hospital's rule we have
\[
p=\lim_{u\to\8}\frac{u}{f(u)/f'(u)}=\lim_{u\to\8}\frac{1}{1-\frac{f(u)f''(u)}{f'(u)^2}}=\frac{q}{q-1},\ \textrm{and hence}\ \frac{1}{p}+\frac{1}{q}=1.
\]
The exponent $q$ represents the H\"older conjugate of the growth rate of $f$.
We will see in Section~2 that $q\ge 1$.
Let $q_S$ and $q_{JL}$ denote the H\"older conjugates of the critical Sobolev exponent and the so-called Joseph-Lundgren exponent, respectively, i.e.,
\[
q_S:=\frac{N+2}{4}\ \ \textrm{and}\ \ q_{JL}:=\frac{N-2\sqrt{N-1}}{4}.
\]
The exponents $q_S$ and $q_{JL}$ can be formally defined for all $N\ge 1$.
However, $q_S>1$ (resp. $q_{JL}>1$) if and only if $N\ge 3$ (resp. $N\ge 11$).
In Section~2 we give five examples of $f$.
In particular, $q=p/(p-1)$ if $f(u)=u^p$, and $q=1$ if $f(u)=e^u$.
The case $q=1$ includes rapidly growing nonlinearities, e.g.,
\[
\textrm{$\exp(u^p)$ ($p\ge 1$), $\exp(\exp(\cdots\exp(u)\cdots))$, and $\underbrace{(u+2)^{\left[ (u+2)^{\left[\rddots {(u+2)}\right]}\right]}}_{n}$ $(n\ge 2)$.}
\]
When $q>1$, the principal term of $f$ is not necessarily $u^{q/(q-1)}$, e.g., $f(u)=u^{q/(q-1)}(\log (u+1))^{\gamma}$ ($\gamma>0$).

The first main result of the paper is the following:
\begin{maintheorem}\label{THA}
Suppose that $N\ge 3$ and (\ref{f1}) and (\ref{f2}) hold.
Let $u(r,\rho)$ be the solution of (\ref{S1E2}).
Then the following hold:\\
(i) If $q<q_S$, then there is $r_0>0$ such that (\ref{S1EQ}) has a singular solution $u^*(r)\in C^2(0,r_0)$ and
\begin{equation}\label{THAE1}
u^*(r)=F^{-1}[k^{-1}r^2(1+\theta(r))].
\end{equation}
Here, $\theta(r)\in C^2(0,r_0)$, $\theta(r)\to 0$ $(r\downarrow 0)$, and
\begin{equation}\label{THAE1+}
k:=2N-4q.
\end{equation}
(ii) If $q_{JL}<q<q_S$, then for each $r_1\in (0,r_0]$, $\calZ_{(0,r_1)}[u(\,\cdot\,,\rho)-u^*(\,\cdot\,)]\to\8$ as $\rho\to\8$. Here $r_0$ is given in (i).
\end{maintheorem}

In order to prove Theorem~\ref{THA} we use the following generalized similarity transformation of elliptic type:
\begin{equation}\label{S1E3}
v(s,\sigma)=F^{-1}[\lambda^{-2}F[u(r,\rho)]]\ \ \textrm{and}\ \ s:=\frac{r}{\lambda},
\end{equation}
where $\sigma:=F^{-1}[\lambda^{-2}F(\rho)]$.
The parabolic version of (\ref{S1E3}) was introduced by Fujishima~\cite[the equation (11)]{F14}.
We show in Section~4 that a certain limit equation for $v$ becomes
\begin{equation}\label{LE}
v''+\frac{N-1}{s}v'+f(v)+\frac{q-F(v)f'(v)}{F(v)f(v)}v'^2=0
\end{equation}
when $N\ge 3$ and $q<q_S$.
If $v(s)$ satisfies (\ref{LE}), then $F^{-1}[\lambda^{-2} F(v(\lambda s))]$, $\lambda>0$, also satisfies (\ref{LE}).
The equation (\ref{LE}) has the exact singular solution
\begin{equation}\label{SSV}
v^*(s):=F^{-1}[k^{-1}s^2].
\end{equation}
Let $v(s,\sigma)$ denote the solution of the problem
\begin{equation}\label{LP}
\begin{cases}
v''+\frac{N-1}{s}v'+f(v)+\frac{q-F(v)f'(v)}{F(v)f(v)}v'^2=0, & s>0,\\
v(0)=\sigma,\\
v'(0)=0.
\end{cases}
\end{equation}
A large solution of (\ref{S1EQ}) is approximated by a solution of (\ref{LE}).
Thus, it is important to study the intersection number of two solutions of (\ref{LE}).
The second main result is the following:
\begin{maintheorem}\label{THB}
Suppose that $N\ge 3$ and (\ref{f1}) and (\ref{f2}) hold.
Let $0<\sigma_0<\sigma_1$.
Let $v(s,\sigma_i)$, $i=0,1$, be solutions of (\ref{LP}) with $\sigma=\sigma_i$, and let $v^*(s)$ be the singular solution given by (\ref{SSV}).
Then the following hold:\\
(i) If $q=q_S$, then $\calZ_{(0,\8)}[v(\,\cdot\,,\sigma_0)-v(\,\cdot\,,\sigma_1)]=1$ and $\calZ_{(0,\8)}[v(\,\cdot\,,\sigma_0)-v^*(\,\cdot\,)]=2$.\\
(ii) If $q_{JL}<q<q_S$, then $\calZ_{(0,\8)}[v(\,\cdot\,,\sigma_0)-v(\,\cdot\,,\sigma_1)]=\8$ and $\calZ_{(0,\8)}[v(\,\cdot\,,\sigma_0)-v^*(\,\cdot\,)]=\8$.\\
(iii) If $q\le q_{JL}$, then $\calZ_{(0,\8)}[v(\,\cdot\,,\sigma_0)-v(\,\cdot\,,\sigma_1)]=0$ and $\calZ_{(0,\8)}[v(\,\cdot\,,\sigma_0)-v^*(\,\cdot\,)]=0$.\\
In particular, if $3\le N\le 9$, then $q_{JL}<1$, and hence (iii) is vacuous, because $q\ge 1$.
\end{maintheorem}
When $q=q_S$, $v(s,\sigma)$ can be written explicitly as follows:
\[
v(s,\sigma)=F^{-1}\left[F(\sigma)\left(1+\frac{s^2}{4NF(1)}\right)^2\right].
\]
For the case of a quasilinear elliptic equation with power nonlinearity, see \cite{MT16}.
\begin{remark}\label{R1}
When $q=1$, the condition $q_{JL}<q<q_S$ corresponds to $3\le N\le 9$.
\end{remark}

As we will see in Subsections~2.1 and 2.2, Theorem~\ref{THB} is a generalization of the well-known result about intersection numbers for the case $f(u)=u^p$ or $e^u$.
In this paper two applications of Theorems~\ref{THA} and \ref{THB} are given in Corollaries~\ref{S1C0} and \ref{S1C1} below.\\

The first application is about the Morse index of a singular solution.
When (\ref{f1}) and (\ref{f2}) hold, (\ref{S1EQ}) has a singular solution $u^*(r)$ given by Theorem~\ref{THA}.
The solution $u^*(r)$ is defined near the origin.
We extend the domain of $u^*(r)$.
We assume that $u^*(r)$ has a first positive zero $r^*_0>0$, i.e., $u^*(r)>0$ for $0<r<r^*_0$ and $u^*(r^*_0)=0$.
For example, if $f(u)>\delta>0$ for $u\ge 0$, then $u^*(r)$ has the first positive zero.
The function $u^*(r)$ is a singular solution of the Dirichlet problem
\begin{equation}\label{DP}
\begin{cases}
\Delta u+f(u)=0, & \textrm{in}\ B(r^*_0),\\
u=0, & \textrm{on}\ \partial B(r^*_0),\\
u>0, & \textrm{in}\ B(r^*_0),
\end{cases}
\end{equation}
where $B(r^*_0)\subset\R^N$ is a ball with radius $r^*_0$.
The Morse index of $u^*(r)$ is defined by
\[
m(u^*):=\sup\left\{\dim X;\ \textrm{$X$ is a subspace of $H^1_{0,rad}(B(r^*_0))$},\ H[\phi]<0\ \textrm{for all $\phi\in X\backslash\{0\}$}\right\},
\]
where $H^1_{0,rad}(B(r^*_0)):=\{u(x)\in H^1(B(r^*_0));\ u(x)=u(|x|),\ u(r^*_0)=0\}$ and
\[
H[\phi]:=\int_{B(r^*_0)}\left(|\nabla\phi|^2-f'(u^*)\phi^2\right)dx.
\]
\begin{corollary}[Morse index of the singular solution]\label{S1C0}
Suppose that $N\ge 3$ and (\ref{f1}) and (\ref{f2}) hold.
Let $u^*(r)$ be the singular solution of (\ref{DP}) constructed as above.
Then the following hold:\\
(i) If $q_{JL}<q<q_S$, then $m(u^*)=\8$.\\
(ii) If $q<q_{JL}$, then $m(u^*)<\8$.
\end{corollary}
If $3\le N\le 9$, then (ii) does not occur.
We need a more explicit description of $\theta(r)$ in (\ref{THAE1}) to determine the case $q=q_{JL}$.
The intersection number is used in the proof of (i).
Let $p_S$ and $p_{JL}$ be the critical Sobolev exponent and the Joseph-Lundgren exponent, respectively, i.e.,
\[
p_S:=
\begin{cases}
\frac{N+2}{N-2} & \textrm{if}\ N\ge 3,\\
\8 & \textrm{if}\ N=2
\end{cases}
\quad\textrm{and}\quad
p_{JL}:=
\begin{cases}
1+\frac{4}{N-4-2\sqrt{N-1}} & \textrm{if}\ N\ge 11,\\
\infty & \textrm{if}\ 2\le N\le 10.
\end{cases}
\]
When $f(u)=u^p+c_0u$ with $p>p_S$, a singular solution $u^*$ of (\ref{DP}) was constructed by Merle-Peletier~\cite{MP91}.
Guo-Wei~\cite{GW11} studied the Morse index of this singular solution $u^*$.
In \cite{GW11} it was shown that if $p_S<p<p_{JL}$ (resp. $p\ge p_{JL}$), then $m(u^*)=\8$ (resp. $m(u^*)<\8$).
Note that their result includes the case $p=p_{JL}$ which corresponds to the case $q=q_{JL}$.
In author's previous papers~\cite{Mi14,Mi15} singular solutions of (\ref{DP}) were constructed for the case $f(u)=u^p+g_0(u)$ and $f(u)=e^u+g_1(u)$, where $g_0$ and $g_1$ are lower order terms.
Moreover, Corollary~\ref{S1C0} was proved for these two cases.
When $f(u)=\exp(u^p)$ ($p>0$), Kikuchi-Wei~\cite{KW16} constructed a singular solution $u^*$ of (\ref{DP}) and showed that $m(u^*)<\8$ if $N\ge 11$.

We have to mention the uniqueness of a radial singular solution in the supercritical case.
The uniqueness was proved by several authors. See \cite{CCCT10,FF15,HKW13,LL10,MN16,WCCK14}.
However, the equations treated by them are $\Delta u+a_0(r)u^p+g_0(u,r)=0$ and $\Delta_m u+a_1(r)u^p+g_1(u,r)=0$ ($1<m\le 2$), where $\Delta_m$ denotes the $m$-Laplace operator and $g_0$ and $g_1$ are lower order terms.
For other supercritical nonlinearities the uniqueness problem remains open.\\

The second application is a bifurcation problem.
Let $B\subset\R^N$ denote the unit ball.
We consider the problem
\begin{equation}\label{BP}
\begin{cases}
\Delta U+\mu f(U)=0, & \textrm{in}\ B,\\
U=0, & \textrm{on}\ \partial B,\\
U>0, & \textrm{in}\ B,
\end{cases}
\end{equation}
where $\mu>0$.
We assume the following:
\begin{equation}\label{f1'}
f\in C^1[0,\8)\cap C^2(1,\8),\ \ f(u)>0\ \textrm{for}\ u\ge 0,\ \ \textrm{and}\ \ \lim_{u\to\8}F(u)=0.\tag{f1'}
\end{equation}
When (\ref{f1'}) holds, the domain of $f$ can be extended to $\R$ such that (\ref{f12}) holds.
See Section~8 for details.
By the symmetry result of Gidas-Ni-Nirenburg~\cite{GNN79} every classical solution of (\ref{BP}) is radial.
Hence, (\ref{BP}) can be reduced to an ODE.
It is well known that the solution set of (\ref{BP}) can be described as $\{(\mu(\rho),U(R,\rho))\}$ and $\rho:=\left\|U\right\|_{L^{\infty}(B)}$ and that $\mu(0)=0$.
Hence, the solution set is a curve emanating from $(\lambda,U)=(0,0)$.
See \cite{K97} for example.
\begin{corollary}[Bifurcation diagram]\label{S1C1}
Suppose that $N\ge 3$, $q<q_S$, and (\ref{f1'}) and (\ref{f2}) hold.
Then, (\ref{BP}) has a singular solution $(\mu^*,U^*)$ and the following hold:\\
(i) If $q_{JL}<q<q_S$, then the curve $\{(\mu(\rho),U(R,\rho));\ \rho>0\}$ has infinitely many turning points around $\mu^*$.
In particular, (\ref{BP}) has infinitely many classical solutions for $\mu=\mu^*$.\\
(ii) If $q\le q_{JL}$, $f''(u)>0$ for $u\ge 0$, and $q\le F(u)f'(u)\le (N-2)^2/(8(N-2q))$ for $u\ge 0$, then $\mu(\rho)$ is strictly increasing, and hence it has no turning point.
The curve can be parametrized by $\mu$ and it blows up at some $\bar{\mu}>0$.
Therefore, (\ref{BP}) has a unique classical solution for each $0<\mu<\bar{\mu}$.
\end{corollary}
The intersection number is used in the proof of Corollary~\ref{S1C1}~(i) and (ii).
If $3\le N\le 9$, (\ref{f1'}) and (\ref{f2}) hold, and $q=1$, then $q_{JL}<q<q_S$, and hence the conclusions of Corollary~\ref{S1C1}~(i) hold.
When $f(u)=(u+1)^p$ ($p_S<p<p_{JL}$) and $f(u)=e^u$ ($3\le N\le 9$), Joseph-Lundgren~\cite{JL73} proved Corollary~\ref{S1C1} by phase plane analysis.
See Jacobsen-Schmitt~\cite{JS02} for quasilinear equations with $f(u)=e^u$.
When $f(u)=u^p+c_0u$ ($p_S<p<p_{JL}$), the existence of infinitely many turning points was numerically shown by Budd-Norbury~\cite{BN87}, and later proved by Dolbeault-Flores~\cite{DF07} and Guo-Wei~\cite{GW11}.
In \cite{Mi14,Mi15} Corollary~\ref{S1C1}~(i) was proved in the case $f(u)=u^p+g_0(u)$ ($p_S<p<p_{JL}$) and $f(u)=e^u+g_1(u)$ ($3\le N\le 9$), where $g_0$ and $g_1$ are lower order terms.
When $f(u)=\exp(u^p)$, $p>0$, $3\le N\le 9$, and the domain is not necessarily a ball, Dancer~\cite{D13} proved the existence of infinitely many turning points.
However, the locations of the turning points are not determined, and hence the existence of infinitely many positive solutions for some $\mu>0$ remains open.
In the ball case Kikuchi-Wei~\cite{KW16} proved Corollary~\ref{S1C1}~(i), using another similarity transformation introduced by \cite{D13}.
On the other hand, a non-existence of a turning point was studied by Brezis-V\'azquez~\cite{BV97} in a general domain.
They gave a necessary and sufficient condition, using a singular solution.
See \cite{Mi14,Mi15,MN16} for the cases $f(u)=u^p+g_0(u)$ and $f(u)=e^u+g_1(u)$ in the ball.\\

Let us explain technical details.
The exponent $q$, which is the H\"older conjugate of the growth rate of $f$, was introduced in Dupaigne-Farina~\cite{DF10}.
They gave sufficient conditions for (\ref{S1E1}) to have a bounded stable nonnegative solution in $\R^N$, using $q$.
Fujishima-Ioku~\cite{FI16} used the exponent $A:=\lim_{u\to\8}F(u)f'(u)$ to study the solvablity of semilinear parabolic equations.
When $f$ is a $C^2$-function, $A$ is equal to $q$.
Indeed, by L'Hospital's rule we have
\[
A=\lim_{u\to\8}F(u)f'(u)=\lim_{u\to\8}\frac{F(u)}{1/f'(u)}=\lim_{u\to\8}\frac{f'(u)^2}{f(u)f''(u)}=q.
\]
One of the advantages of using $q$ is that $q$ can deal with the exponential and super-exponential growth ($q=1$), while $p=q/(q-1)$ cannot.
A singular solution of Theorem~\ref{THA}~(i) is constructed by a standard method with the contraction mapping theorem.
However, our method can be applied to rather general nonlinearities owing to the expression of the singular solution (\ref{THAE1}).
In the proof of Theorem~\ref{THA}~(ii) we rescale a regular and singular solutions of (\ref{S1EQ}), using the generalized similarity transformation (\ref{S4E1}).
Then these two functions locally uniformly converge to a regular and singular solutions of (\ref{LE}), respectively.
Hence, (\ref{LE}) can be considered as a limit equation of (\ref{S1EQ}).
By $f_q(u)$ we define
\[
f_q(u):=
\begin{cases}
u^p & \textrm{if}\ q>1,\\
e^u & \textrm{if}\ q=1,
\end{cases}
\]
where $p:=q/(q-1)$ provided that $q>1$.
Let $v(s,\sigma)$ be the solution of (\ref{LP}).
We use a generalized Cole-Hopf transformation:
\begin{equation}\label{S1E5}
w(s,\tau):=F_q^{-1}[F[v(s,\sigma)]]\ \ \textrm{and}\ \ \tau:=F_q^{-1}[F[\sigma]],
\end{equation}
where
\[
F_q[v]:=\int_v^{\infty}\frac{dt}{f_q(t)}=
\begin{cases}
\frac{1}{p-1}v^{-p+1} & \textrm{if}\ q>1,\\
e^{-v} & \textrm{if}\ q=1
\end{cases}
\]
and $F_q^{-1}$ is the inverse function of $F_q$.
Specifically,
\[
w(s,\tau):=
\begin{cases}
(p-1)^{\frac{-1}{p-1}}\left(F[v(s,\sigma)]\right)^{\frac{-1}{p-1}} & \textrm{if}\ q>1,\\
-\log F[v(s,\sigma)] & \textrm{if}\ q=1,
\end{cases}
\]
where $p:=q/(q-1)$ if $q>1$.
This transformation was introduced by Fujishima-Ioku~\cite{FI16} and used in the study of the existence of a solution for semilinear parabolic equations.
By Lemma~\ref{S5L0} we see that $w$ satisfies $\Delta w+f_q(w)=0$, i.e.,
\begin{equation}\label{OPF}
\begin{cases}
\Delta w+w^p=0 & \textrm{if}\ q>1,\\
\Delta w+e^w=0 & \textrm{if}\ q=1.
\end{cases}
\end{equation}
It is interesting that all the limit equations (\ref{LE}) can be classified into two typical cases (\ref{OPF}) and that all the limit equations become a one-parameter family of equations $\{\Delta w+f_q(w)=0;\ q\ge 1\}$ in spite of arbitrariness of $f$.
See Figure~\ref{fig1} which shows a fundamental strategy in the paper.
Intersection properties of the two typical equations are well known.
Theorem~\ref{THB} follows from them.
\begin{figure}\label{fig1}
\begin{tabular}{ccl}
$\Delta v+f(v)+\frac{q-F(v)f'(v)}{F(v)f(v)}|\nabla v|^2=0$ & 
$\xLongleftrightarrow[\textrm{equivalent}]{{F[v(s)]=F_q[w(s)]}}$ &
$\Delta w+f_q(w)=0$
\\
\mbox{$\quad\qquad$}\rotatebox[origin=c]{90}{$\longrightarrow$} $\textrm{scaling limit as $\lambda\downarrow 0$}\atop{(\lambda^2F[v(s)]=F[u(r)],\ \lambda s=r)}$
& & 
\mbox{$\quad\qquad$}\rotatebox[origin=c]{90}{$\longleftrightarrow$} $\textrm{scale invariant(self-similar)}\atop{(\lambda^2F_q[w(s)]=F_q[u(r)],\ \lambda s=r)}$
\\
$\Delta u+f(u)=0\ \ \textrm{with $1\le q<q_S$}$ & & 
$\Delta u+f_q(u)=0$
\end{tabular}\caption{Relation between the original equation and limit equation}
\end{figure}

Corollary~\ref{S1C0} is a simple application of Theorem~\ref{THA}.
In the proof of Corollary~\ref{S1C0} we use a convexity of $f$.
In the proof of Corollary~\ref{S1C1}~(i) we use Theorem~\ref{THA}~(ii).
Specifically, the existence of infinitely many turning points corresponds to $\calZ_{(0,r_1)}[u(\,\cdot\,,\rho)-u^*(\,\cdot\,)]\to\8$ as $\rho\to\8$.
In the proof of Corollary~\ref{S1C1}~(ii) we use a comparison theorem devised by Gui~\cite{G96,GNW92}.
We can compare radial solutions of (\ref{LE}) and $\Delta u+f_q(u)=0$.
Then we show that the first (and hence every) eigenvalue is strictly positive, using an argument similar to the one used in Brezis-V\'azquez~\cite{BV97}.
Therefore, the curve has no turning point.
These methods were used in \cite{KW16,Mi14,Mi15,MN16} although nonlinearities are restricted.

This paper consists of eight sections.
In Section~2 we show that $q\ge 1$.
Five examples of nonlinearities are given.
In Section~3 we construct a singular solution.
In Section~4 we prove the convergence to solutions of the limit equation of (\ref{S1EQ}).
In Section~5 and 6 we prove Theorem~\ref{THB} and \ref{THA}, respectively.
In Sections~7 and 8 we prove Corollaries~\ref{S1C0} and \ref{S1C1}, respectively.

\section{Preliminaries and five nonlinearities}
The following lemma is a fundamental property of the exponent $q$, which was found by \cite{FI16}.
See \cite[Remark 1.1]{FI16}.
\begin{lemma}\label{S2L1}
Suppose that (\ref{f1}) and (\ref{f2}) hold.
Then $q\ge 1$.
\end{lemma}
We show the proof for readers' convenience.
\begin{proof}
By L'Hospital's rule we have
\[
\lim_{u\to\8}F(u)f'(u)=\lim_{u\to\8}\frac{F(u)}{{1}/{f'(u)}}=\lim_{u\to\8}\frac{f'(u)^2}{f(u)f''(u)}=q.
\]
We show by contradiction that $q\ge 1$.
Suppose that $0\le q<1$.
There exists $\kappa\in (0,1)$ such that $F(u)f'(u)<\kappa$ for large $u>0$.
Since $f'(u)=F''(u)F(u)^{-2}$, there is $u_0\in\R$ such that $F(u)F''(u)F'(u)^{-2}<\kappa$ for $u>u_0$.
Note that $F'(u)<0$.
Solving the differential inequality, we have
\[
F(u)^{1-\kappa}<F(u_0)^{1-\kappa}-(1-\kappa)\frac{-F'(u_0)}{F(u_0)^{\kappa}}(u-u_0)\ \ \textrm{for}\ \ u>u_0.
\]
We see that $F(u)<0$ for large $u>0$.
We obtain a contradiction, because $F(u)>0$.
\end{proof}

\subsection{Power nonlinearity, {\boldmath $q>1$}}
Let $f_a(u):=(u+a)^p$, $p>1$.
Then, (\ref{f1'}) (resp. (\ref{f11})) holds if $a>0$ (resp. $a=0$).
Since $F^{-1}[\lambda^{-2}F(u(\lambda s))]=\lambda^{2/(p-1)}(u(\lambda s)+a)-a$, the transformation $u(s)\mapsto F^{-1}[\lambda^{-2}F(u(\lambda s))]$ is the usual similarity transformation in the power case when $a=0$.
Since $f_a'(u)^2/(f_a(u)f''_a(u))\equiv p/(p-1)$ for $u\ge 0$, we see that $q=p/(p-1)$, and hence (\ref{f2}) holds.
When $a=0$, Theorem~\ref{THB} recovers the well-known intersection property studied by \cite{JL73,W93}.
Since $F(u)f'_a(u)\equiv q$ for $u\ge 0$, the limit equation (\ref{LE}) is the original equation (\ref{S1EQ}).
Therefore, if $u(r)$ is a solution of (\ref{S1EQ}), then $F^{-1}[\lambda^{-2}F(u(\lambda r))]$ is also a solution of (\ref{S1EQ}).
When $a>0$, Corollary~\ref{S1C0} is applicable.
If $p_S<p<p_{JL}$, then $q_{JL}<q<q_S$ and Corollary~\ref{S1C1}~(i) is applicable.
If $p\ge p_{JL}$, then $q\le q_{JL}$ and Corollary~\ref{S1C1}~(ii) is applicable.

\subsection{Exponential nonlinearity, {\boldmath $q=1$}}
Let $f(u):=e^u$.
Then (\ref{f12}) holds.
Since $F^{-1}[\lambda^{-2}F(u(\lambda s))]=u(\lambda s)+2\log\lambda$, the transformation $u(s)\mapsto F^{-1}[\lambda^{-2}F(u(\lambda s))]$ is the usual transformation in the exponential case.
Since $f'(u)^2/(f(u)f''(u))\equiv 1$ for $u\in\R$, we see that $q=1$, and hence (\ref{f2}) holds.
Theorem~\ref{THB} recovers the intersection property for the case $f(u)=e^u$, which was studied by \cite{JL73}.
Since $F(u)f'(u)\equiv 1$ for $u\ge 0$, the limit equation (\ref{LE}) is the original equation (\ref{S1EQ}).
Therefore, if $u(r)$ is a solution of (\ref{S1EQ}), 
$F^{-1}[\lambda^{-2}F(u(\lambda r))]$ is also a solution of (\ref{S1EQ}).
Corollary~\ref{S1C0} is applicable.
If $3\le N\le 9$, then $q_{JL}<q<q_S$ and Corollary~\ref{S1C1}~(i) is applicable.
If $N\ge 10$, then $q\le q_{JL}$ and Corollary~\ref{S1C1}~(ii) is applicable.

\subsection{Log-convex or log-concave function, {\boldmath $q=1$}}
We consider the case $f(u):=\exp(g(u))$ and $g'(u)>0$ $(u\ge 0)$.
Since
\[
F(u)f'(u)=1-g'(u)e^{g(u)}\int_u^{\infty}\frac{g''(t)}{g'(t)^2}e^{-g(t)}dt
\]
and
\begin{equation}\label{S23E1}
\frac{f'(u)^2}{f(u)f''(u)}=\frac{1}{1+\frac{g''(u)}{g'(u)^2}}\,
\end{equation}
we can check that if $g''(u)<0$, $g'(u)^2+g''(u)>0$, $\lim_{u\to\infty}g''(u)/g'(u)^2=0$, and $N>0$ is large, then Corollary~\ref{S1C1}~(ii) is applicable.

Next we consider the case $f(u):=\exp(g(u))$ and $g'(u)>0$ $(u\ge 0)$ and $g''(u)>0$ $(u\ge 0)$.
The following lemma holds:
\begin{lemma}\label{S2L2}
If $f$ satisfies (\ref{f2}), then $q=1$.
\end{lemma}
\begin{proof}
We see by (\ref{S23E1}) that $\lim_{u\to\8}{f'(u)^2}/{f(u)f''(u)}\le 1$, because the limit exists.
By Lemma~\ref{S2L1} we see that $q=1$.
\end{proof}
We can easily construct an example of $g(u)$ such that $g'(u)>0$, $g''(u)>0$, and $g''(u)/g'(u)^2$ oscillates as $u\to\8$.
Therefore, the limit $\lim_{u\to\8}g''(u)/g'(u)^2$ may not exist even if $g'(u)>0$ and $g''(u)>0$.
The assumption (\ref{f2}) in Lemma~\ref{S2L2} cannot be removed.

We give some examples such that the limit exists and Corollaries~\ref{S1C0} and \ref{S1C1}~(i) are applicable to $f(u)=\exp(g(u))$.
\begin{enumerate}
\item $f(u)=\exp(u^p)$ ($p>1$),
\item $f(u)=\exp(e^u)$,
\item If $g_n(u)$ satisfies the following:
\begin{equation}\label{S2S3E0}
\textrm{For $u\ge 0$, $g_n'(u)>0$ and $g_n''(u)>0$, and $\lim_{u\to\8}\frac{g_n''(u)}{g_n'(u)^2}=0$},
\end{equation}
then $g_{n+1}(u):=\exp(g_n(u))$ also satisfies (\ref{S2S3E0}).
Thus, $f_n(u):=\exp(g_{n}(u))$, $n\ge 1$, satisfies (\ref{f1'}) and (\ref{f2}) if (\ref{S2S3E0}) holds for $n=1$.
A typical example is the case $g_2(u):=e^u$.
Then $f_n$ is the iterated exponential function
\[
\textrm{$f_n(u):=\underbrace{\exp(\exp(\cdots\exp(u)\cdots))}_{n}$ $(n\ge 2)$}.
\]
\end{enumerate}

\subsection{Product of power and log, {\boldmath $q>1$}}
Let $f(u):=(u+a)^p(\log (u+a))^{\gamma}$, $p>1$, $\gamma\ge 0$, and $a>1$.
Then (\ref{f1'}) holds.
By direct calculation we see that $q=\lim_{u\to\8}f'(u)^2/(f(u)f''(u))=p/(p-1)$.
Corollaries~\ref{S1C0} and \ref{S1C1}~(i) are applicable.

\subsection{Tetration, {\boldmath $q=1$}}
Let $f_{n+1}(u):=(u+a)^{f_n(u)}$ $(a>1)$.
Suppose the following:
\begin{equation}\label{S25E1}
\textrm{For $u\ge 0$, $f_n(u)>1$, $f'_n(u)>0$, and $f''_n(u)>0$ and}\ \lim_{u\to\8}\frac{f'_n(u)^2}{f_n(u)f''_n(u)}=1.
\end{equation}
Then we easily see that $\lim_{u\to\8}f_n(u)=\8$.
Using this limit, we can show that $f_{n+1}(u)$ also satisfies (\ref{S25E1}).
By (\ref{S25E1}) and the definition of $f_{n+1}(u)$ we have that
\begin{equation}\label{S25E2}
\frac{f'_{n+2}(u)^2}{f_{n+2}(u)f''_{n+2}(u)}=
\left[ 1+\frac{f''_{n+1}(u)\log (u+a)+{2f'_{n+1}(u)}{(u+a)^{-1}}
-{f_{n+1}(u)}{(u+a)^{-2}}}{
 \left\{f'_{n+1}(u)\log(u+a)+{f_{n+1}(u)}{(u+a)^{-1}}
 \right\}^2}
\right]^{-1}\le 1,
\end{equation}
because
\[
f''_{n+1}(u)\log (u+a)+\frac{2f'_{n+1}(u)}{u+a}-\frac{f_{n+1}(u)}{(u+a)^2}\ge 0\ \ \textrm{for}\ \ u\ge 0.
\]
By (\ref{S25E1}) and (\ref{S25E2}) we see that if $f_1(u)$ satisfies (\ref{S25E1}), then $f_n(u)$ satisfies (\ref{f1'}) and (\ref{f2}) for $n\ge 1$.

Let $f_2(u):=(u+a)^{u+a}$ $(a>1)$ and $f_{n+1}:=(u+a)^{f_n(u)}$ $(n\ge 2)$.
Then,
\[
\textrm{$f_n(u)=\underbrace{(u+a)^{\left[ (u+a)^{\left[\rddots {(u+a)}\right]}\right]}}_{n}$}
\]
which is called the $n$-th tetration of $(u+a)$.
Corollaries~\ref{S1C0} and \ref{S1C1}~(i) are applicable to $f_n(u)$, $n\ge 2$.

\section{Singular solution}
The goal of this section is to prove the following:
\begin{lemma}\label{S3L1}
Suppose that $N\ge 3$ and (\ref{f1}) and (\ref{f2}) hold.
Let $k$ be defined by (\ref{THAE1+}).
If $q<q_S$, then there are a small $r_0>0$ and $\theta(r)\in C^2(0,r_0)$ such that $u^*(r):=F^{-1}[k^{-1}r^2(1+\theta(r))]$ satisfies (\ref{S1EQ}) on $(0,r_0)$ and $\lim_{r\downarrow 0}\theta(r)=0$.
As a consequence, $\lim_{r\downarrow 0}u^*(r)=\8$, and hence $u^*(r)$ is a singular solution of (\ref{S1EQ}) near $r=0$.
\end{lemma}
\begin{proof}
We find a singular solution of the form (\ref{THAE1}).
Let $T<0$ be negatively large, and let $x(t)\in C^2(-\8,T+1)$.
We assume that $\left\|x(t)\right\|_{L^{\8}(-\8,T+1)}$ is small.
We set $x(t):=\theta(r)$ and $t:=\log r$.
Then
\begin{equation}\label{L1E0}
u^*(r)=F^{-1}[k^{-1}e^{2t}(1+x(t))].
\end{equation}
Substituting $u^*(r)$ into (\ref{S1EQ}), we have
\[
x''+(N+2)x'+2Nx+4q-q\frac{(x'+2x+2)^2}{x+1}-(F(u^*)f'(u^*)-q)\frac{(x'+2x+2)^2}{x+1}=0.
\]
The equation is equivalent to
\[
x''+(N+2-4q)x'+(2N-4q)x-\frac{qx'^2}{x+1}-(F(u^*)f'(u^*)-q)\frac{(x'+2x+2)^2}{x+1}=0.
\]
We construct a solution such that $x(t)\to 0$ as $t\to-\8$.
Let $y(t):=x'(t)$. Then
\[
\frac{d}{dt}\left(
\begin{array}{c}
x\\
y
\end{array}
\right)=-A\left(
\begin{array}{c}
x\\
y
\end{array}
\right)+\left(
\begin{array}{c}
0\\
f_0(x,y)+f_1(x,y,t)
\end{array}
\right),
\]
where
\begin{align*}
A:=&\left(
\begin{array}{cc}
0 & -1\\
2N-4q & N+2-4q
\end{array}
\right),\\
f_0(x,y)&:=\frac{qy^2}{x+1},\\
f_1(x,y,t)&:=(F(u^*)f'(u^*)-q)\frac{(y+2x+2)^2}{x+1}.
\end{align*}
We show that the Lipschitz constants of $f_0$ and $f_1$ are small.
Let $X:=C((-\8,T],\R^2)$, and let $\e>0$ be small.
Here, $T<0$ and $\e>0$ are chosen later.
We define $B_{\e}:=\{(x,y)\in X;\ \left\|(x,y)\right\|_X:=\left\|x\right\|_{L^{\8}(-\8,T]}+\left\|y\right\|_{L^{\8}(-\8,T]}<\e\}$.
Let $(x_1,y_1)$, $(x_2,y_2)\in B_{\e}$.
We have
\begin{align}
|f_0(x_2,y_2)-f_0(x_1,y_1)|&\le q\left|\frac{y_1+y_2}{x_2+1}\right|\left|y_2-y_1\right|+\frac{q|y_1|^2}{|(x_1+1)(x_2+1)|}\left|x_2-x_1\right|\nonumber\\
&\le C\e(|x_2-x_1|+|y_2-y_1|).\label{L1E1}
\end{align}
By $u_i(r)$, $i=1,2$, we define $u_i(r):=F^{-1}[k^{-1}e^{2t}(1+x_i(t))]$.
We have
\begin{align}
|f_1(x_2,y_2,t)-f_1(x_1,y_1,t)|&\le
\left|F(u_1)f'(u_1)-q\right|\left|\frac{(2x_2+y_2+2)^2}{x_2+1}-\frac{(2x_1+y_1+2)^2}{x_1+1}\right|\nonumber\\
&\quad+\left|F(u_2)f'(u_2)-F(u_1)f'(u_1)\right|\left|\frac{(2x_2+y_2+2)^2}{x_2+1}\right|\label{L1E1+}.
\end{align}
Since
\begin{align}
\left|\frac{(2x_2+y_2+2)^2}{x_2+1}-\frac{(2x_1+y_1+2)^2}{x_1+1}\right|
&\le\frac{|2x_2+y_2+2|^2}{|(x_1+1)(x_2+1)|}|x_2-x_1|\nonumber\\
&\quad+\frac{|2x_1+2x_2+y_1+y_2+4|}{|x_1+1|}|2(x_2-x_1)+(y_2-y_1)|\nonumber\\
&\le C(|x_2-x_1|+|y_2-y_1|)\ \ \textrm{and}\nonumber
\end{align}
$|F(u_1)f'(u_1)-q|<\e$ for $t<T$, we have
\begin{equation}\label{L1E2}
|F(u_1)f'(u_1)-q|\left|\frac{(2x_2+y_2+2)^2}{x_2+1}-\frac{(2x_1+y_1+2)^2}{x_1+1}\right|\le C\e(|x_2-x_1|+|y_2-y_1|)
\end{equation}
for $t<T$.
Let $w_i(t):=F(u_i(r))$, $i=1,2$.
Since
\begin{align*}
\frac{d}{dw}(wf'(F^{-1}(w))) &=\left(1-wf'(F^{-1}(w))\frac{f(F^{-1}(w))f''(F^{-1}(w))}{f'(F^{-1}(w))^2}\right)f'(F^{-1}(w)),
\end{align*}
it follows from the mean value theorem that there is $\bw$ such that
\begin{multline*}
w_2f'(F^{-1}(w_2))-w_1f'(F^{-1}(w_1))\\
=\left(1-\bw f'(F^{-1}(\bw))\frac{f(F^{-1}(\bw))f''(F^{-1}(\bw))}{f'(F^{-1}(\bw))^2}\right)f'(F^{-1}(\bw))(w_2-w_1),
\end{multline*}
and $\min\{w_1(t),w_2(t)\}\le\bw(t)\le\max\{w_1(t),w_2(t)\}$.
Let $\bu:=F^{-1}(\bw)$. Then
\begin{equation}\label{L1E3}
F(u_2)f'(u_2)-F(u_1)f'(u_1)=\left(1-F(\bu)f'(\bu)\frac{f(\bu)f''(\bu)}{f'(\bu)^2}\right)f'(\bu)(F(u_2)-F(u_1)),
\end{equation}
where $\bu(r)$ satisfies $\min\{u_1(r),u_2(r)\}\le\bu(r)\le\max\{u_1(r),u_2(r)\}$.
Since $\lim_{r\downarrow 0}\min\{u_1(r),u_2(r)\}=\8$, $\lim_{r\downarrow 0}\bu(r)=\8$.
By L'Hospital's rule we have
\begin{equation}\label{L1E3+}
\lim_{u\to\8}F(u)f'(u)=\lim_{u\to\8}\frac{\int_{u}^{\8}{ds}/{f(s)}}{{1}/{f'(u)}}=\lim_{u\to\8}\frac{f'(u)^2}{f(u)f''(u)}=q.
\end{equation}
Thus,
\[
\lim_{u\to\8}F(u)f'(u)\frac{f(u)f''(u)}{f'(u)^2}=1.
\]
This means that
\[
\left|1-F(\bu)f'(\bu)\frac{f(\bu)f''(\bu)}{f'(\bu)^2}\right|<\e\ \ \textrm{for small $r>0$.}
\]
Since $\min\{u_1(r),u_2(r)\}\le\bu(r)\le\max\{u_1(r),u_2(r)\}$, there is $\bx(t)$ such that $\min\{x_1(t),x_2(t)\}\le\bx(t)\le\max\{x_1(t),x_2(t)\}$ and $\bu(r)=F^{-1}(k^{-1}e^{2t}(1+\bx(t)))$.
Then $ke^{2t}(1+\bx)=F(\bu)$.
We have
\begin{align}
|f'(\bu)(F(u_2)-F(u_1))|&\le|f'(\bu)ke^{2t}(x_2-x_1)|\nonumber\\
&\le\left|f'(\bu)F(\bu)\frac{x_2-x_1}{\bx+1}\right|\nonumber\\
&\le C|F(\bu)f'(\bu)||x_2-x_1|.\label{L1E4}
\end{align}
By (\ref{L1E3}) and (\ref{L1E4}) we see that
\begin{equation}\label{L1E5}
|F(u_2)f'(u_2)-F(u_1)f'(u_1)|\le C\e |x_2-x_1|\ \ \textrm{for $t<T$.}
\end{equation}
By (\ref{L1E1+}), (\ref{L1E2}), and (\ref{L1E5}) we have
\begin{equation}\label{L1E6}
|f_1(x_2,y_2,t)-f_1(x_1,y_1,t)|\le C\e (|x_2-x_1|+|y_2-y_1|)\ \ \textrm{for $t<T$.}
\end{equation}

By $\xi(t)$, $G(\xi(t),t)$, $\calF(\xi(t))$ we define
\begin{align*}
\xi(t)&:=\left(
\begin{array}{c}
x(t)\\
y(t)
\end{array}
\right),\\
G(\xi(t),t)&:=\left(
\begin{array}{c}
0\\
f_0(x,y)+f_1(x,y,t)
\end{array}
\right),\\
\calF(\xi(t))&:=\int_{-\8}^te^{-(t-\tau)A}G(\xi(\tau),\tau)d\tau,
\end{align*}
respectively.
We find a solution of the equation $\xi(t)=\calF(\xi(t))$ in $B_{\e}$ if $\e>0$ is small and $T<0$ is negatively large.
Then the solution $\xi(t)$ is corresponding to a solution of (\ref{S1EQ}) near $r=0$.
In fact, $u^*(r)=F^{-1}[k^{-1}r^2(1+x(\log r))]$ becomes a singular solution near $r=0$.
The eigenvalues of $A$ are
\[
\lambda_{\pm}:=\frac{(N+2-4q)\pm\sqrt{D}}{2},\ \ \textrm{where}\ \ D:=(N+2-4q)^2-4(2N-4q).
\]
Therefore, $\left\| e^{-tA}\right\|_{\R^2}\le Ce^{-\mu t}$, where
\[
\mu:=
\begin{cases}
\frac{N+2-4q-\sqrt{D}}{2}>0 & \textrm{if}\ 1\le q<q_{JL},\\
\frac{N+2-4q}{2}-\delta>0 & \textrm{if}\ q=q_{JL},\\
\frac{N+2-4q}{2}>0 & \textrm{if}\ q_{JL}<q<q_S,
\end{cases}
\]
and $\delta>0$ is small.
If $t<T$, then by (\ref{L1E1}) and (\ref{L1E6}) we have
\begin{align*}
\left\|\calF(\xi_2(t))-\calF(\xi_1(t))\right\|_{\R^2}&\le
\int_{-\8}^t\left\|e^{-(t-\tau)A}(G(\xi_2(\tau),\tau)-G(\xi_1(\tau),\tau))\right\|_{\R^2}d\tau\\
&\le C\int_{-\8}^te^{-\mu(t-\tau)}d\tau C\e\left\|\xi_2(t)-\xi_1(t)\right\|_X\\
&=\frac{C\e}{\mu}\left\|\xi_2(t)-\xi_1(t)\right\|_X.
\end{align*}
If $T<0$ is negatively large, then $\e>0$ can be chosen arbitrarily small.
Hence, there is $\kappa\in(0,1)$ such that $\left\|\calF(\xi_2)-\calF(\xi_1)\right\|_X\le\kappa\left\|\xi_2-\xi_1\right\|_X$ for $\xi_1$, $\xi_2\in B_{\e}$.
Since $f_0(0,0)=0$ and $|f_1(0,0,t)|=o(1)$ as $t\to -\infty$, we see that $\left\|\calF(0)\right\|_X=o(1)$.
Then,
\begin{equation}\label{L1E7}
\left\|\calF(\xi)\right\|_X\le\left\|\calF(\xi)-\calF(0)\right\|_X+\left\|\calF(0)\right\|_X\le\kappa\e+o(1)<\e
\end{equation}
provided that $T<0$ is negatively large.
Hence, $\calF$ is a contraction mapping on $B_{\e}$.
It follows from the contraction mapping theorem that $\calF$ has a unique fixed point in $B_{\e}$ which is a solution of $\xi=\calF(\xi)$.
When $T$ is negatively large, $\e>0$ can be taken small.
By (\ref{L1E7}) and the uniqueness of the fixed point in $B_{\e}$ we see that $\left\|\xi(t)\right\|_{\R^2}\to 0$ as $t\to -\8$.
Thus, $x(t)\to 0$ as $t\to -\8$.
Let $\theta(r):=x(\log r)$.
Then the conclusion holds.
\end{proof}

\section{Convergence to a solution of the limit problem}
Let $u(r,\rho)$ be the solution of (\ref{S1E2}).
By $\tu(s)$ we define
\begin{equation}\label{S4E1}
\tu(s):=F^{-1}[\lambda^{-2}F[u(r,\rho)]],\ \ s:=\frac{r}{\lambda},\ \ \textrm{and}\ \ \lambda:=\sqrt{\frac{F(\rho)}{F(1)}}.
\end{equation}
Then $\tu(s)$ satisfies
\begin{equation}\label{S4E0}
\begin{cases}
\tu''(s)+\frac{N-1}{s}\tu'(s)+f(\tu(s))+\frac{F(u(\lambda s,\rho))f'(u(\lambda s,\rho))-F(\tu(s))f'(\tu(s))}{F(\tu(s))f(\tu(s))}\tu'(s)^2=0, & s>0,\\
\tu(0)=1,\\
\tu'(0)=0.
\end{cases}
\end{equation}
First, we study the limit of $\tilde{u}(s)$ as $\rho\to\8$.
Note that $\lambda\downarrow 0$ as $\rho\to\8$.
We need the following proposition:
\begin{proposition}[{\cite[Theorem 2.4]{L94}}]\label{S4P1}
Let $h\in C^1(0,\8)\cap C[0,\8)$.
Assume that $N\ge 3$ and there exists $p>(N+2)/(N-2)$ such that
\begin{equation}\label{S4P1E1}
uh(u)\ge\left(1+p\right)H(u)\ \ \textrm{for large}\ \ u>0,
\end{equation}
where $H(u):=\int_0^uh(t)dt$.
Let $u(r,\rho)$ be the solution of the problem
\[
\begin{cases}
u''+\frac{N-1}{r}u'+h(u)=0, & r>0,\\
u(0)=\rho,\\
u'(0)=0,
\end{cases}
\]
and let $R(\rho)$ be the first positive zero of $u(\,\cdot\,,\rho)$ if it exists.
Then there are $\rho_0\in\R$ and $R_0>0$ such that $R(\rho)>R_0$ for $\rho>\rho_0$.
\end{proposition}

\begin{lemma}\label{S4L1}
Suppose that $N\ge 3$ and (\ref{f1}) and (\ref{f2}) with $q<q_S$ holds.
Let $v(s,1)$ be the solution of (\ref{LP}) with $\sigma=1$. Then
\[
\tu(s)\to v(s,1)\ \ \textrm{in}\ \ C_{loc}[0,\8)\ \ \textrm{as}\ \ \rho\to\8.
\]
\end{lemma}
\begin{proof}
Let $s_0>0$ be fixed.
We show that
\begin{equation}\label{S4L1E1}
u(\lambda s,\rho)\to\8\ \ \textrm{uniformly in}\ s\in [0,s_0]\ \textrm{as}\ \rho\to\8.
\end{equation}
Let $M>0$ be large.
Let $u_M(r):=u(r,\rho)-M$.
Then $u_M$ satisfies
\[
\begin{cases}
u_M''+\frac{N-1}{r}u_M'+f_M(u_M)=0, & r>0,\\
u_M(0)=\rho-M,\\
u_M'(0)=0,
\end{cases}
\]
where $f_M(u):=f(u+M)$.
Because of (\ref{f1}) and (\ref{f2}), we easily see that $f(u)\to\8$ as $u\to\8$.
There is $c>0$ such that $f_M(u)>c$ for $u\ge 0$.
Hence, $u_M$ has the first positive zero when $\rho>M$.
Let $F_0(u):=\int_0^uf(t)dt$.
We consider the case $1<q<q_S$.
Using L'Hospital rule twice, we have
\[
\lim_{u\to\8}\frac{uf(u)}{F_0(u)}=\lim_{u\to\8}\left(1+\frac{u}{\frac{f(u)}{f'(u)}}\right)=\lim_{u\to\8}\left(1+\frac{1}{1-\frac{f(u)f''(u)}{f'(u)^2}}\right)=1+\frac{q}{q-1}.
\]
Since $1<q<q_S$, we see that $1+q/(q-1)>1+(N+2)/(N-2)$.
Therefore, $f_M$ satisfies (\ref{S4P1E1}) in Proposition~\ref{S4P1} if $M>0$ is large.
We consider the case $q=1$.
Then $\lim_{u\to\8}uf(u)/F_0(u)=\8$, and hence $f_M$ satisfies (\ref{S4P1E1}).
In both cases we see by Proposition~\ref{S4P1} that there are $\rho_M>0$ and $r_M>0$ such that if $\rho>\rho_M$, then $u(r,\rho)\ge M$ for $0\le r\le r_M$.
If $\rho>0$ is sufficiently large, then $u(\lambda s,\rho)\ge M$ for $0\le s\le s_0$, because $\lim_{\rho\to\8}\lambda=0$ and $s_0<r_M/\lambda$.
Since $M>0$ can be chosen arbitrarily large, (\ref{S4L1E1}) is proved, and hence,
\[
F(u(\lambda s,\rho))f'(u(\lambda s,\rho))\to q\ \ \textrm{uniformly in}\ s\in [0,s_0]\ \textrm{as}\ \rho\to\8.
\]
Because of the continuity of $\tu(s)\in C[0,s_0)$ with respect to the nonlinearity in (\ref{S4E0}), we see that $\tu(s)\to v(s,1)$ in $C[0,s_0)$ as $\rho\to\8$.
Since $s_0>0$ can be chosen arbitrarily, we obtain the conclusion.
\end{proof}
Because of Lemma~\ref{S4L1}, the limit equation of (\ref{S1EQ}) as $\rho\to\infty$ becomes (\ref{LE}) when $u$ and $v$ are related by (\ref{S1E3}) with $\lambda=\sqrt{F(\rho)/F(1)}$.

Second, we study the limit of the rescaled singular solution as $\lambda\downarrow 0$.
\begin{lemma}\label{S4L2}
Suppose that $N\ge 3$ and (\ref{f1}) and (\ref{f2}) with $q<q_S$ hold.
Let $u^*(r)$ be the singular solution given by Lemma~\ref{S3L1}, and let $s$ and $\lambda$ be defined by (\ref{S4E1}).
Let $\tu^*(s):=F^{-1}[\lambda^{-2}F[u^*(r)]]$. Then
\[
\tu^*(s)\to v^*(s)\ \ \textrm{in}\ \ C_{loc}(0,\8)\ \ \textrm{as}\ \ \rho\to\8,
\]
where $v^*(s)=F^{-1}[k^{-1}s^2]$ which is defined by (\ref{SSV}).
\end{lemma}
\begin{proof}
Since $u^*(r)=F^{-1}[k^{-1}r^2(1+\theta(r))]$, we have
\[
\tu^*(s)=F^{-1}[\lambda^{-2}F[F^{-1}[k^{-1}(\lambda s)^2(1+\theta(\lambda s))]]]=F^{-1}[k^{-1}s^2(1+\theta(\lambda s))].
\]
Let $0<s_0<s_1$ be fixed.
Since $\theta(r)\to 0$ as $r\downarrow 0$, we see that $\theta(\lambda s)\to 0$ uniformly in $s\in [s_0,s_1]$ as $\rho\to\8$.
Thus, $\tu^*(s)\to F^{-1}[k^{-1}s^2]$ uniformly in $s\in[s_0,s_1]$ as $\rho\to\8$.
Since $s_0$ and $s_1$ can be chosen arbitrarily under the condition $0<s_0<s_1$, we obtain the conclusion.
\end{proof}
Because of Lemmas~\ref{S4L1} and \ref{S4L2}, the limit functions of $\tu(s)$ and $\tu^*(s)$ satisfy (\ref{LE}).

\section{Proof of Theorem~\ref{THB}}
Let $v(s,\sigma)$ be the solution of (\ref{LP}), and let $w(s,\tau)$ and $\tau$ be defined by (\ref{S1E5}).
Note that $\tau=F_q^{-1}[F[\sigma]]$ and $\tau$ is strictly increasing in $\sigma$.
\begin{lemma}\label{S5L0}
Suppose that $N\ge 3$.
The following two statements hold:\\
(i) If $q>1$, then $w$ satisfies
\begin{equation}\label{S5E1}
\begin{cases}
w''+\frac{N-1}{s}w'+w^p=0, & s>0,\\
w(0)=\tau,\\
w'(0)=0
\end{cases}
\end{equation}
and 
\begin{equation}\label{S5E1+}
w^*(s):=F_q^{-1}[F[v^*(s)]]=\left\{\frac{2}{p-1}\left(N-2-\frac{2}{p-1}\right)\right\}^{1/(p-1)}s^{-\frac{2}{p-1}}
\end{equation}
is a singular solution of the equation in (\ref{S5E1}).
Here, $p:=q/(q-1)$.\\
(ii) If $q=1$, then $w$ satisfies
\begin{equation}\label{S5E2}
\begin{cases}
w''+\frac{N-1}{s}w'+e^w=0, & s>0,\\
w(0)=\tau,\\
w'(0)=0
\end{cases}
\end{equation}
and
\begin{equation}\label{S5E2+}
w^*(s):=F_q^{-1}[F[v^*(s)]]=-2\log s+\log 2(N-2)
\end{equation}
is a singular solution of the equation in (\ref{S5E2}).
\end{lemma}
\begin{proof}
By direct calculation we can obtain the conclusions.
We omit the proof.
\end{proof}
The following two propositions are well known:
\begin{proposition}\label{S5P1}
Suppose that $N\ge 3$.
Let $0<\tau_0<\tau_1$.
Let $w(s,\tau_i)$, $i=0,1$, be solutions of (\ref{S5E1}) with $\tau=\tau_i$, and let $w^*(s)$ be the singular solution given by (\ref{S5E1+}).
Then the following holds:\\
(i) If $p=p_S$, then $\calZ_{(0,\8)}[w(\,\cdot\,,\tau_0)-w(\,\cdot\,,\tau_1)]=1$ and $\calZ_{(0,\8)}[w(\,\cdot\,,\tau_0)-w^*(\,\cdot\,)]=2$.\\
(ii) If $p_{S}<p<p_{JL}$, then $\calZ_{(0,\8)}[w(\,\cdot\,,\tau_0)-w(\,\cdot\,,\tau_1)]=\8$ and $\calZ_{(0,\8)}[w(\,\cdot\,,\tau_0)-w^*(\,\cdot\,)]=\8$.\\
(iii) If $p\ge p_{JL}$, then $\calZ_{(0,\8)}[w(\,\cdot\,,\tau_0)-w(\,\cdot\,,\tau_1)]=0$ and $\calZ_{(0,\8)}[w(\,\cdot\,,\tau_0)-w^*(\,\cdot\,)]=0$.
\end{proposition}
See \cite{JL73,MT16,W93} for example.

\begin{proposition}\label{S5P2}
Suppose that $N\ge 3$.
Let $\tau_0<\tau_1$.
Let $w(s,\tau_i)$, $i=0,1$, be solutions of (\ref{S5E2}) with $\tau=\tau_i$, and let $w^*(s)$ be the singular solution given by (\ref{S5E2+}).
Then the following holds:\\
(i) If $3\le N\le 9$, then $\calZ_{(0,\8)}[w(\,\cdot\,,\tau_0)-w(\,\cdot\,,\tau_1)]=\8$ and $\calZ_{(0,\8)}[w(\,\cdot\,,\tau_0)-w^*(\,\cdot\,)]=\8$.\\
(ii) If $N\ge 10$, then $\calZ_{(0,\8)}[w(\,\cdot\,,\tau_0)-w(\,\cdot\,,\tau_1)]=0$ and $\calZ_{(0,\8)}[w(\,\cdot\,,\tau_0)-w^*(\,\cdot\,)]=0$.
\end{proposition}
See \cite{JS02,JL73} for example.

\begin{proof}[Proof of Theorem~\ref{THB}]
We consider the case $q>1$.
The cases $q=q_S$, $q_{JL}<q<q_S$, and $q\le q_{JL}$ correspond to $p=p_S$, $p_S<p<p_{JL}$, and $p\ge p_{JL}$, respectively.
The conclusion of Theorem~\ref{THB} follows from Proposition~\ref{S5P1}.
We consider the case $q=1$.
The cases $q_{JL}<q<q_S$ and $q\le q_{JL}$ correspond to $3\le N\le 9$ and $N\ge 10$, respectively.
The conclusion of Theorem~\ref{THB} follows from Proposition~\ref{S5P2}.
Note that the case $q=q_S$ corresponds to $N=2$.
The proof is complete.
\end{proof}

\section{Proof of Theorem~\ref{THA}}
\begin{proof}[Proof of Theorem~\ref{THA}]
(i) follows from Lemma~\ref{S3L1}.

(ii) Let $\tu(s):=F^{-1}[\lambda^{-2}F[u(r,\rho)]]$, $\tu^*(s):=F^{-1}[\lambda^{-2}F[u^*(r)]]$, $s:=r/\lambda$, and $\lambda:=\sqrt{F(\rho)/F(1)}$.
By Lemmas~\ref{S4L1} and \ref{S4L2} we see that as $\rho\to\8$,
\begin{equation}\label{PTHAE1}
\tu(s)\to v(s,1)\ \ \textrm{in}\ \ C_{loc}[0,\8),\\
\end{equation}
\begin{equation}\label{PTHAE2}
\tu^*(s)\to v^*(s)\ \ \textrm{in}\ \ C_{loc}(0,\8).
\end{equation}
Since $q_{JL}<q<q_S$, Theorem~\ref{THB}~(ii) says that
\begin{equation}\label{PTHAE3}
\calZ_{(0,\8)}[v(\,\cdot\,,1)-v^*(\,\cdot\,)]=\8.
\end{equation}
Let $r_1>0$ be given by the assumption in Theorem~\ref{THA}.
Since the same transformation is applied to $\tu(s)$ and $\tu^*(s)$, we see that $\calZ_{(0,r_1)}[u(\,\cdot\,,\rho)-u^*(\,\cdot\,)]=\calZ_{(0,r_1/\lambda)}[\tu(\,\cdot\,)-\tu^*(\,\cdot\,)]$.
For each $M>0$, there is $s_M>0$ and $\rho_M>0$ such that $\calZ_{(0,s_M)}[\tu(\,\cdot\,)-\tu^*(\,\cdot\,)]\ge M$ for $\rho>\rho_M$, because of (\ref{PTHAE1}), (\ref{PTHAE2}) and (\ref{PTHAE3}).
If $\rho>0$ is large, then $(0,s_M)\subset (0,r_1/\lambda)$, and hence
\[
\calZ_{(0,r_1)}[u(\,\cdot\,,\rho)-u^*(\,\cdot\,)]=\calZ_{(0,r_1/\lambda)}[\tu(\,\cdot\,)-\tu^*(\,\cdot\,)]\ge\calZ_{(0,s_M)}[\tu(\,\cdot\,)-\tu^*(\,\cdot\,)]\ge M.
\]
Since $M$ can be chosen arbitrarily large, $\calZ_{(0,r_1)}[u(\,\cdot\,,\rho)-u^*(\,\cdot\,)]\to\8$ as $\rho\to\8$.
\end{proof}
\section{Morse index of a singular solution}
\begin{proof}[Proof of Corollary~\ref{S1C0}]
(i) Because of (\ref{f2}), $f''(u)>0$ for large $u>0$, and hence $f(u)$ is convex for large $u$.
There is $r_2\in (0,r_0]$ such that $f(u)$ is convex for $u\in\{u^*(r)\in\R;\ 0<r<r_2\}$.
Here, $r_0$ is given in Theorem~\ref{THA}~(i).
It follows from Theorem~\ref{THA}~(ii) that, for each $n\ge 1$, there is a large $\rho>0$ such that $u^*(\,\cdot\,)-u(\,\cdot\,,\rho)$ has at least $2n+1$ zeros in $(0,r_2)$.
Then $\{z_i\}_{i=1}^{2n+1}$, $z_1<z_2<\cdots<z_{2n+1}$, denotes the first $2n+1$ zeros.
By $\phi_i(r)$, $i=1,2,\ldots,n$, we define
\[
\phi_i(r):=
\begin{cases}
u^*(r)-u(r,\rho) & \textrm{if}\ r\in(z_{2i},z_{2i+1}),\\
0 & \textrm{if}\ r\not\in(z_{2i},z_{2i+1}).
\end{cases}
\]
Since $f(u)$ is convex for large $u$ and $u(r,\rho)<u^*(r)$ for $r\in(z_{2i},z_{2i+1})$, we have that
\[
\frac{f(u^*(r))-f(u(r,\rho))}{u^*(r)-u(r,\rho)}<f'(u^*(r))\ \ \textrm{for}\ \ r\in(z_{2i},z_{2i+1}).
\]
Let
\[
V(r):=
\begin{cases}
\frac{f(u^*(r))-f(u(r,\rho))}{u^*(r)-u(r,\rho)} & \textrm{if}\ u^*(r)\neq u(r,\rho),\\
f'(u^*(r)) & \textrm{if}\ u^*(r)=u(r,\rho).
\end{cases}
\]
Since $\phi_i(z_{2i})=\phi_i(z_{2i+1})=0$, we have that
\begin{align}
\int_{B(r^*_0)}\left(|\nabla\phi_i|^2-f'(u^*)\phi_i^2\right)dx
&=\int_{\{z_{2i}<r<z_{2i+1}\}}\left(|\nabla\phi_i|^2-f'(u^*)\phi_i^2\right)dx\nonumber\\
&<\int_{\{z_{2i}<r<z_{2i+1}\}}\left(|\nabla\phi_i|^2-V\phi_i^2\right)dx\nonumber\\
&=-\int_{\{z_{2i}<r<z_{2i+1}\}}\phi_i\left(\Delta\phi_i+V\phi_i\right)dx\nonumber\\
&\quad+\int_{\{r=z_{2i+1}\}}\phi_i\frac{\partial}{\partial {\rm n}}\phi_id\sigma
+\int_{\{r=z_{2i}\}}\phi_i(-\frac{\partial}{\partial {\rm n}}\phi_i')d\sigma\nonumber\\
&=0,\label{PC0E0}
\end{align}
where ${\partial}/{\partial {\rm n}}$ denotes the outer normal derivative, and we use $\Delta\phi_i+V\phi_i=0$ in $\{z_{2i}<r<z_{2i+1}\}$.
Since the supports of $\phi_i$ and $\phi_j$, $j\neq i$, are disjoint and $n$ can be chosen arbitrarily large, (\ref{PC0E0}) indicates that $m(u^*)=\infty$.

(ii)  Let $\e>0$ be fixed.
We determine $\e$ later.
Since $\lim_{u\to\8}F(u)f'(u)=q$ (See (\ref{L1E3+})), we see that $F(u^*(r))f'(u^*(r))<q+\e$ for small $r>0$.
By Theorem~\ref{THA}~(i) we have $F(u^*(r))=k^{-1}r^2(1+\theta(r))$.
Therefore, $f'(u^*(r))<2(N-2q)(q+\e)/(r^2(1+\theta(r)))$ for small $r>0$.
Since $(q+\e)/(1+o(1))<q+2\e$ for small $\e>0$, there is $r_3>0$ such that
\begin{equation}\label{PC0E-0}
f'(u^*(r))<\frac{2(N-2q)(q+2\e)}{r^2}\ \ \textrm{for}\ \ 0<r<r_3.
\end{equation}
Since $q<q_{JL}$, we see that $2(N-2q)q<(N-2)^2/4$.
We can choose $\e>0$ small enough so that $2(N-2q)(q+2\e)<(N-2)^2/4$.
Then there is $\delta\in(0,1)$ such that
\begin{equation}\label{PC0E-1}
2(N-2q)(q+2\e)=(1-\delta)\frac{(N-2)^2}{4}.
\end{equation}
We define
\[
\chi_0(r):=
\begin{cases}
1 & \textrm{if}\ 0\le r<r_3/2,\\
0 & \textrm{if}\ r_3<r,
\end{cases}
\]
where $0\le\chi_0(r)\le 1$ and $\chi_0(r)\in C^1$.
Let $\chi_1(r):=1-\chi_0(r)$.
By (\ref{PC0E-0}) and (\ref{PC0E-1}) we see that for $\phi\in H^1_{0,rad}(B(r^*_0))$,
\begin{align}
\int_{B(r^*_0)}\left(\left|\nabla\phi\right|^2-f'(u)\phi^2\right)dx
&=\int_{B(r^*_0)}\left\{(1-\delta)\left|\nabla\phi\right|^2-\chi_0f'(u^*)\phi^2\right\}dx\nonumber\\
&\quad+\int_{B(r^*_0)}\left\{\delta\left|\nabla\phi\right|^2-\chi_1f'(u^*)\phi^2\right\}dx\nonumber\\
&\ge (1-\delta)\int_{B(r^*_0)}\left(\left|\nabla\phi\right|^2-\frac{(N-2)^2}{4}\phi^2\right)dx\nonumber\\
&\quad+\delta\int_{B(r^*_0)}\left(\left|\nabla\phi\right|^2-\frac{\chi_1}{\delta}f'(u^*)\phi^2\right)dx\nonumber\\
&\ge\delta\int_{B(r^*_0)}\left(\left|\nabla\phi\right|^2-\frac{\chi_1}{\delta}f'(u^*)\phi^2\right)dx,\label{PC0E-2}
\end{align}
where we used Hardy's inequality.
Since $|\chi_1f'(u^*)/\delta|$ is bounded on $B(r^*_0)$, the operator $-\Delta-\chi_1f'(u^*)/\delta$ with the Dirichlet boundary condition has at most finitely many negative eigenvalues, i.e., $\dim X_1<\8$.
Here,
\[
X_1:=\left\{\phi\in H^1_{0,rad}(B(r^*_0));\ \int_{B(r^*_0)}\left(\left|\nabla\phi\right|^2-\frac{\chi_1}{\delta}f'(u^*)\phi^2\right)dx<0\right\}\cup\{0\}.
\]
We prove $m(u^*)<\8$ by contradiction.
Suppose that $m(u^*)=\8$.
Then
\[
\dim\left(\left\{\phi\in H^1_{0,rad}(B(r^*_0));\ \int_{B(r^*_0)}\left(\left|\nabla\phi\right|^2-f'(u^*)\phi^2\right)dx<0\right\}\cup\{0\}\right)=\8.
\]
Because of (\ref{PC0E-2}), we see that $\dim X_1=\8$.
We obtain a contradiction.
Thus, $m(u^*)<\8$.
\end{proof}

\section{Bifurcation diagram}
Suppose that $N\ge 3$ and (\ref{f1'}) and (\ref{f2}) with $q<q_S$ hold.
The problem (\ref{BP}) is equivalent to the following problem
\[
\begin{cases}
U''+\frac{N-1}{R}U'+\mu f(U)=0 & 0<R<1,\\
U(1)=0,\\
U(R)>0, & 0<R<1.
\end{cases}
\]
Let $u(r):=U(R)$ and $r:=\sqrt{\mu}R$.
Then $u$ satisfies
\[
\begin{cases}
u''+\frac{N-1}{r}u'+f(u)=0, & 0<r<\sqrt{\mu},\\
u(\sqrt{\mu})=0,\\
u(r)>0, & 0<r<\sqrt{\mu}.
\end{cases}
\]
We consider (\ref{S1E2}).
Let $u(r,\rho)$ be the solution of (\ref{S1E2}).
Because of (\ref{f1'}), there is $\delta>0$ such that
\begin{equation}\label{S8E1}
f(u)>\delta\ \ \textrm{for}\ \ u\ge 0.
\end{equation}
It is well known that $u(\,\cdot\,,\rho)$ has the first positive zero $r_0(\rho)$.
Let $\mu(\rho):=r_0(\rho)^2$.
Because of Theorem~\ref{THA}~(i), (\ref{S1EQ}) has a singular solution $u^*(r)$ near $r=0$.
We extend the domain of $u^*$.
By (\ref{S8E1}) we see that $u^*(r)$ also has the first positive zero $r^*_0$.
Then $(\mu^*,U^*(R)):=((r^*_0)^2,u^*(r^*_0R))$ is a singular solution of (\ref{BP}).

We extend the domain of $f(u)$.
We can assume that $f\in C^1(\R)$, $f(u)>0$ for $u\in\R$, and $f(u)=\delta/2$ for $u<-1$.
Then (\ref{f12}) holds, and hence Theorem~\ref{THA} is applicable.
\begin{lemma}\label{S8L1}
Suppose that $N\ge 3$ and (\ref{f1'}) and (\ref{f2}) hold.
Let $u(r,\rho)$ be the solution of (\ref{S1E2}) and $\mu^*:=(r_0^*)^2$.
Let $\mu(\rho)$ be as above.
If $q_{JL}<q<q_S$, then $\mu(\rho)$ oscillates infinitely many times around $\mu^*$ as $\rho\to\8$.
\end{lemma}
\begin{proof}
Let $z(\rho):=\calZ_{(0,\min\{r_0(\rho),r_0^*\})}[u(\,\cdot\,,\rho)-u^*(\,\cdot\,)]$ and let $I:=(0,\min\{r_0(\rho),r_0^*\})$.
It is clear that $\{u(r,\rho)-u^*(\,\cdot\,)=0\}$ does not have an accumulation point.
Hence, $z(\rho)<\infty$.
We see that each zero of $u(r,\rho)-u^*(r)$ is simple and $u(r,\rho)-u^*(r)$ is a $C^1$-function of $(r,\rho)$.
It follows from the implicit function theorem that each zero of $u(r,\rho)-u^*(r)$ continuously depends on $\rho$.
Because $z(\rho)$ does no change in a neighborhood of each fixed $\rho$, $z(\rho)$ does not change if another zero does not enter $I$ from $\partial I$ or a zero in $I$ goes out of $I$.
We prove the conclusion of the lemma by contradiction.
Suppose that there is $\rho_0>0$ such that $\mu(\rho)<\mu^*$ for $\rho>\rho_0$.
Let $\tilde{r}:=\min\{r_0(\rho_0),r_0^*\}$.
We see that $u(0,\rho)-u^*(0)=-\infty$ and $u(\tilde{r},\rho)-u^*(\tilde{r})<0$.
Thus, another zero cannot enter, and $z(\rho)$ is bounded for $\rho>0$.
This contradicts Theorem~\ref{THA}~(ii).
Similarly, we obtain the contradiction in the case where $\mu(\rho)>\mu^*$ for $\rho>\rho_0$.
As a consequence, $\mu(\rho)$ has to oscillate infinitely many times around $\mu^*$ as $\rho\to\8$.
\end{proof}
Corollary~\ref{S1C1}~(i) immediately follows from Lemma~\ref{S8L1}.\\

We prove Corollary~\ref{S1C1}~(ii).
We apply a generalized Cole-Hopf transformation, which is mentioned in Section~1, to $u(r,\rho)$.
Let
\begin{equation}\label{S8E2}
\tilde{u}(r,\sigma):=F_q^{-1}[F[u(r,\rho)]]\ \ \textrm{and}\ \ \sigma:=F_q^{-1}[F(\rho)].
\end{equation}
Note that $\sigma$ is strictly increasing in $\rho$.
The function $\tilde{u}$ satisfies
\[
\Delta\tilde{u}+f_q(\tilde{u})+(F(u)f'(u)-q)|\nabla\tilde{u}|^2=0.
\]
First we show the following:
\begin{lemma}\label{S8L2}
Suppose that $N\ge 3$ (\ref{f1'}) and (\ref{f2}) with $q\le q_{JL}$ hold, and $F(u)f'(u)\ge q$ for $u\ge 0$.
Then,
\begin{equation}\label{S8L2E0}
\tilde{u}(r,\sigma)\le F_q^{-1}[k^{-1}r^2]\ \ \textrm{for}\ \ r\ge 0.
\end{equation}
\end{lemma}

\begin{proof}
Let $\alpha_1>\alpha_0(>\sigma)$ and let $\tilde{u}_i$, $i=0,1$, be the solution of the problem
\[
\begin{cases}
\tilde{u}_i''+\frac{N-1}{r}\tilde{u}_i'+f_q(\tilde{u}_i')=0, & r>0,\\
\tilde{u}_i(0)=\alpha_i,\\
\tilde{u}_i'(0)=0.
\end{cases}
\]
Since $q\le q_{JL}$, Propositions~\ref{S5P1} and \ref{S5P2} say that $\tilde{u}_1(r)>\tilde{u}_0(r)$ for $r\ge 0$.
Let $w_1(r):=\tilde{u}_1(r)-\tilde{u}_0(r)$.
Then
\[
\begin{cases}
\Delta w_1+V_1w_1=0 & \textrm{in}\ \R^N,\\
w_1>0 & \textrm{in}\ \R^N,
\end{cases}
\]
where
\[
V_1:=\frac{f_q(\tilde{u}_1(r))-f_q(\tilde{u}_0(r))}{\tilde{u}_1(r)-\tilde{u}_0(r)}.
\]
We show by contradiction that
\begin{equation}\label{S8L2E1}
\tilde{u}(r,\sigma)<\tilde{u}_0(r)\ \ \textrm{for}\ \ r>0.
\end{equation}
Suppose the contrary, i.e., there is $r_0>0$ such that $\tilde{u}(r,\sigma)<\tilde{u}_0(r)$ for $0<r<r_0$ and $\tilde{u}(r_0,\sigma)=\tilde{u}_0(r_0)$.
Let $w_0(r):=\tilde{u}_0(r)-\tilde{u}(r,\sigma)$.
Then
\[
\begin{cases}
\Delta w_0+V_0w_0=(F(u)f'(u)-q)|\nabla\tilde{u}|^2\ge 0 & \textrm{in}\ B(r_0),\\
w_0>0 & \textrm{in}\ B(r_0),
\end{cases}
\]
where $B(r_0)$ is an open ball with radius $r_0$ and
\[
V_0:=
\begin{cases}
\frac{f_q(\tilde{u}_0(r))-f_q(\tilde{u}(r,\sigma))}{\tilde{u}_0(r)-\tilde{u}(r,\sigma)} & \textrm{if}\ \tilde{u}_0(r)\neq\tilde{u}(r,\sigma),\\
f'_q(\tilde{u}_0(r)) & \textrm{if}\ \tilde{u}_0(r)=\tilde{u}(r,\sigma).
\end{cases}
\]
Since $f_q$ is strictly convex, we see that $V_1>V_0$.
Let $\omega_N$ denote the surface area of the unit sphere $\mathbb{S}^{N-1}\subset\R^N$.
Since $w_0'(r_0)\le 0$ and $w_0(r_0)=0$,
\begin{align*}
0&>-\int_{B(r_0)}(V_1-V_0)w_0w_1+w_1(F(u)f'(u)-q)|\nabla\tilde{u}|^2dx\\
&=\int_{B(r_0)}(w_0\Delta w_1-w_1\Delta w_0)dx\\
&=\omega_Nr_0^{N-1}(w_0(r_0)w_1'(r_0)-w_1(r_0)w_0'(r_0))\ge 0,
\end{align*}
which is a contradiction.
Thus, (\ref{S8L2E1}) holds.
Let $\tilde{u}^*(r):=F_q^{-1}[k^{-1}r^2]$.
Then $\tilde{u}^*$ is a singular solution of $\Delta\tilde{u}^*+f_q(\tilde{u}^*)=0$.
Since $q\le q_{JL}$, Propositions~\ref{S5P1} and \ref{S5P2} say that $\calZ_{(0,\infty)}[\tilde{u}_0(\,\cdot\,)-\tilde{u}^*(\,\cdot\,)]=0$, and hence $\tilde{u}_0(r)<\tilde{u}^*(r)$ for $r>0$.
Thus, (\ref{S8L2E0}) holds.
\end{proof}

\begin{proof}[Proof of Corollary~\ref{S1C1}]
Since $f(0)>0$, the bifurcation curve starts from $(0,0)$ and consists of minimal solutions near $(0,0)$.
Since $f'>0$ and $f''>0$, there is $\bar{\mu}>0$ such that either the curve blows up at $\bar{\mu}$ or the curve has a turning point at $\bar{\mu}$.
See \cite{BCMR96,CR75}.
We show by contradiction that the curve blows up at $\bar{\mu}$.
Suppose the contrary, i.e., $(\bar{\mu},U(R))$ is a turning point.
Let $u(r):=U(R)$ and $r:=\sqrt{\bar{\mu}}R$, and let $\tilde{u}$ be defined by (\ref{S8E2}).
Because of Lemma~\ref{S8L2}, $\tilde{u}(r,\sigma)\le F_q^{-1}[k^{-1}r^2]$.
Hence, $k^{-1}r^2\le F_q[\tilde{u}]=F[u]$.
We have that $u\le F^{-1}[k^{-1}r^2]$.
Using $f''\ge 0$ and the assumption of Corollary~\ref{S1C1}, we have
\begin{align*}
f'(u) &\le f'(F^{-1}[k^{-1}r^2])\\
&\le\frac{(N-2)^2}{8(N-2q)}\frac{1}{F[F^{-1}[k^{-1}r^2]]}\\
&=\frac{(N-2)^2}{4r^2}.
\end{align*}
By Hardy's inequality we have
\[
\int_{B_0}\left(|\nabla\phi|^2-f'(u)\phi^2\right)dx
\ge\int_{B_0}\left(|\nabla\phi|^2-\frac{(N-2)^2}{4r^2}\phi^2\right)dx>0
\]
for $\phi\in H^1_{0,rad}(B_0)\backslash\{0\}$.
Here, $B_0$ is a ball with radius the first positive zero of $u(\,\cdot\,)$.
This inequality indicates that the first eigenvalue of the problem
\[
\begin{cases}
\Delta\phi+f'(u)\phi=-\nu\phi,\\
\phi\in H^1_{0,rad}(B_0)
\end{cases}
\]
is strictly positive.
The first eigenvalue of the eigenvalue problem associated to $(\bar{\mu},U(R))$ is also strictly positive.
We have a contradiction, since $(\bar{\mu},U(R))$ is a turning point.
The proof is complete.
\end{proof}
\bigskip

\noindent
{\bf Acknowledgment.}
The author would like to thank Professor Y.~Fujishima for bringing the transformation (\ref{S1E3}) to his attention and sending him the paper~\cite{FI16}.
He thanks Professor Y.~Naito and Professor N.~Ioku for stimulating discussions.
He also thanks Professor H.~Kikuchi for sending him the paper~\cite{KW16}.
This work was supported by JSPS KAKENHI Grant Number 16K05225.



\begin{thebibliography}{99}
\bibitem{BCMR96}{H.~Brezis, T.~Cazenave, Y.~Martel, and A.~Ramiandrisoa},
{\it Blow up for $u_t-\Delta u=g(u)$ revisited},
Adv. Differential Equations {\bf 1} (1996), 73--90.

\bibitem{BN87}{C.~Budd, C and J.~Norbury},
{\it Semilinear elliptic equations and supercritical growth},
J. Differential Equations {\bf 68} (1987), 169--197.

\bibitem{BV97}{H.~Brezis and J.~V\'azquez},
{\it Blow-up solutions of some nonlinear elliptic problems},
Rev. Mat. Univ. Complut. Madrid {\bf 10} (1997), 443--469.

\bibitem{CCCT10}{J.~Chern, Z.~Chen, J.~Chen, and Y.~Tang},
{\it On the classification of standing wave solutions for the Schr\"odinger equation},
Comm. Partial Differential Equations {\bf 35} (2010), 275--301.

\bibitem{CR75}{M.~Crandall and P.~Rabinowitz},
{\it Some continuation and variational methods for positive solutions of nonlinear elliptic eigenvalue problems},
Arch. Rational Mech. Anal. {\bf 58} (1975), 207--218.

\bibitem{D13}{E.~Dancer},
{\it Some bifurcation results for rapidly growing nonlinearities},
Discrete Contin. Dyn. Syst. {\bf 33} (2013), 153--161.

\bibitem{DF07}{J.~Dolbeault and I.~Flores},
{\it Geometry of phase space and solutions of semilinear elliptic equations in a ball},
Trans. Amer. Math. Soc. {\bf 359} (2007), 4073--4087.

\bibitem{DF10}{L.~Dupaigne and A.~Farina},
{\it Stable solutions of $-\Delta u=f(u)$ in $\R^N$},
J. Eur. Math. Soc. {\bf 12} (2010), 855--882.

\bibitem{FF15}{I.~Flores and M.~Franca},
{\it Phase plane analysis for radial solutions to supercritical quasilinear elliptic equations in a ball},
Nonlinear Anal. {\bf 125} (2015), 128--149.

\bibitem{F14}{Y.~Fujishima},
{\it Blow-up set for a superlinear heat equation and pointedness of the initial data},
Discrete Contin. Dyn. Syst. {\bf 34} (2014), 4617--4645.

\bibitem{FI16}{Y.~Fujishima and N.~Ioku},
{\it Existence and nonexistence of solutions for the heat equation with a superlinear source term},
arXiv:1609.06394.

\bibitem{GNN79}{B.~Gidas, W.~Ni, and  L.~Nirenberg},
{\it Symmetry and related properties via the maximum principle},
Comm. Math. Phys. {\bf 68} (1979), 209--243.

\bibitem{G96}{C.~Gui},
{\it On positive entire solutions of the elliptic equation $\Delta u+K(x)u^p=0$ and its applications to Riemannian geometry},
Proc. Roy. Soc. Edinburgh Sect. A {\bf 126} (1996), 225--237.

\bibitem{GNW92}{C.~Gui, W.-M.~Ni, and X.~Wang},
{\it On the stability and instability of positive steady states of a semilinear heat equation in $\R^n$},
Comm. Pure Appl. Math. {\bf 45} (1992), 1153--1181.

\bibitem{GW11}{Z.~Guo and J.~Wei},
{\it Global solution branch and Morse index estimates of a semilinear elliptic equation with super-critical exponent},
Trans. Amer. Math. Soc. {\bf 363} (2011), 4777--4799.

\bibitem{HKW13}{F.~Hadj Selem, H.~Kikuchi, and J.~Wei},
{\it Existence and uniqueness of singular solution to stationary Schr\"odinger equation with supercritical nonlinearity},
Discrete Contin. Dyn. Syst. {\bf 33} (2013), 4613--4626.

\bibitem{JS02}{J.~Jacobsen and K.~Schmitt},
{\it The Liouville-Bratu-Gelfand problem for radial operators},
J. Differential Equations {\bf 184} (2002), 283--298.

\bibitem{JL73}{D.~Joseph and T.~Lundgren},
{\it Quasilinear Dirichlet problems driven by positive sources},
Arch. Rational Mech. Anal. {\bf 49} (1972/73), 241--269.

\bibitem{KW16}{H.~Kikuchi and J.~Wei},
{\it Bifurcation diagram of solutions to elliptic equation with exponential nonlinearity in higher dimensions},
preprint.

\bibitem{K97}{P.~Korman},
{\it Solution curves for semilinear equations on a ball},
Proc. Amer. Math. Soc. {\bf 125} (1997), 1997--2005.

\bibitem{LL10}{B.~Lai and Q.~Luo},
{\it Uniqueness of singular solution of semilinear elliptic equation},
Proc. Indian Acad. Sci. Math. Sci. {\bf 120} (2010), 583--591.

\bibitem{L94}{S.~Lin},
{\it Positive singular solutions for semilinear elliptic equations with supercritical growth},
J. Differential Equations {\bf 114} (1994), 57--76.


\bibitem{MP91}{F.~Merle and L.~Peletier},
{\it Positive solutions of elliptic equations involving supercritical growth},
Proc. Roy. Soc. Edinburgh Sect. A {\bf 118} (1991), 49--62.

\bibitem{Mi14}{Y.~Miyamoto},
{\it Structure of the positive solutions for supercritical elliptic equations in a ball},
J. Math. Pures Appl. {\bf 102} (2014), 672--701.

\bibitem{Mi15}{Y.~Miyamoto},
{\it Classification of bifurcation diagrams for elliptic equations with exponential growth in a ball},
Ann. Mat. Pura Appl. {\bf 194} (2015), 931--952.

\bibitem{MN16}{Y.~Miyamoto and Y.~Naito},
{\it Uniqueness and extremal properties of singular solutions for supercritical elliptic equations in a ball},
submitted.

\bibitem{MT16}{Y.~Miyamoto and K.~Takahashi},
{\it Generalized Joseph-Lundgren exponent and intersection properties for supercritical quasilinear elliptic equations},
to appear in Arch. Math. (Basel).


\bibitem{W93}{X.~Wang},
{\it On the Cauchy problem for reaction-diffusion equations},
Trans. Amer. Math. Soc. {\bf 337} (1993), 549--590.

\bibitem{WCCK14}{Y.~Wu, Z.~Chen, J.~Chern, and Y.~Kabeya},
{\it Existence and uniqueness of singular solutions for elliptic equation on the hyperbolic space},
Commun. Pure Appl. Anal. {\bf 13} (2014), 949--960.
\end{thebibliography}
\end{document}